\documentclass[a4paper]{article}
\usepackage{mathptmx}
\usepackage{amscd,amssymb,amsmath,amsthm}

\makeatletter

\renewcommand{\@seccntformat}[1]
{{\csname the#1\endcsname}.\hspace{0.3em}}

\renewcommand{\section}{\@startsection
{section}
{1}
{0mm}
{-1.5\baselineskip}
{\baselineskip}
{\bfseries\normalsize}}

\renewcommand{\subsection}{\@startsection
{subsection}
{2}
{0mm}
{-\baselineskip}
{0.5\baselineskip}
{\normalsize\itshape}}

\renewcommand{\subsubsection}{\@startsection
{subsubsection}
{3}
{0mm}
{-.5\baselineskip}
{-2mm}
{\normalsize\itshape}}

\makeatother

\theoremstyle{plain}
\newtheorem*{theorem*}{Theorem}
\newtheorem{theorem}{Theorem}[section]
\newtheorem{lemma}{Lemma}[section]
\newtheorem{corollary}[lemma]{Corollary}
\newtheorem{prop}[lemma]{Proposition}

\newtheorem*{corollary*}{Corollary}

\newtheorem*{tcvbis}{Theorem~\ref{t:cv}bis}

\theoremstyle{definition}
\newtheorem*{defin*}{Definition}
\newtheorem{defin}{Definition}[section]

\theoremstyle{remark}

\newtheorem*{quest*}{Question}

\DeclareMathAlphabet{\matheur}{U}{eur}{m}{n}
\DeclareMathAlphabet{\matheus}{U}{eus}{m}{n}
\DeclareMathAlphabet{\matheuf}{U}{euf}{m}{n}

\numberwithin{equation}{section}

\newcommand{\abs}[1]{\left\lvert#1\right\rvert}

\DeclareMathOperator{\grad}{grad}

\DeclareMathOperator{\Index}{Index}

\begin{document}

\author{Gerasim  Kokarev
\\ {\small\it School of Mathematics, The University of Leeds}
\\ {\small\it Leeds, LS2 9JT, United Kingdom}
\\ {\small\it Email: {\tt G.Kokarev@leeds.ac.uk}}
}

\title{Conformal volume and eigenvalue problems}
\date{}
\maketitle

\begin{abstract}
\noindent
We prove inequalities for Laplace eigenvalues of Riemannian manifolds generalising to higher eigenvalues two classical inequalities for the first Laplace eigenvalue -- the inequality in terms of the $L^2$-norm of mean curvature, due to Reilly in 1977, and the inequality in terms of conformal volume, due to Li and Yau in 1982, and El Soufi and Ilias in 1986. We also obtain bounds for the number of negative eigenvalues of Schr\"odinger operators, and in particular, index bounds for minimal hypersurfaces in spheres.
\end{abstract}

\medskip
\noindent
{\small

}

%
%
%


\section{Statements and discussion of results}
\label{intro}

\subsection{Introduction: classical inequalities for the first Laplace eigenvalue}
Let $(\Sigma^n,g)$ be a closed Riemannian manifold of dimension $n\geqslant 2$. In 1982 Li and Yau~\cite{LiYau82} introduced an important conformal invariant, the so-called $m$-dimensional {\em conformal volume} $V_c(m,\Sigma^n)$ of $\Sigma^n$.  It is defined as the infimum of the conformal volumes $V_c(m,\phi)$ of conformal immersions $\phi:\Sigma^n\to S^m$, where 
\begin{equation}
\label{cv:def}
V_c(m,\phi)=\sup\{\mathit{Vol}(\Sigma^n, (s\circ\phi)^*g_{\mathit{can}})~ |~ s\text{ is a conformal diffeomorphism of }S^m\},
\end{equation}
and $g_{\mathit{can}}$ denotes the canonical round metric on the unit sphere $S^m\subset\mathbb R^{m+1}$. A classical result by Li and Yau~\cite{LiYau82} in dimension $n=2$, and by El Soufi and Ilias~\cite{ESI86} in all dimensions gives the following  bound for the first Laplace eigenvalue of $(\Sigma^n,g)$ in terms of the conformal volume.
\begin{theorem}
\label{ly:theorem}
Let $(\Sigma^n,g)$ be a closed Riemannian manifold of dimension $n\geqslant 2$. Then for any integer $m>0$ such that the $m$-dimensional conformal volume of $\Sigma^n$ is defined, the first Laplace eigenvalue of $\Sigma^n$ satisfies the following inequality
\begin{equation}
\label{ly:1st}
\lambda_1(\Sigma^n,g)\mathit{Vol}_g(\Sigma^n)^{2/n}\leqslant n V_c(m,\Sigma^n)^{2/n}.
\end{equation}
Besides, the equality occurs if and only if after rescaling the metric $g$ the manifold $(\Sigma^n,g)$ admits an isometric minimal immersion into a sphere $S^m$ by first eigenfunctions.
\end{theorem}

The case of equality plays an important role in extremal eigenvalue problems, see~\cite{Na96, Ko14, FS3} and references there, and indicates on an intimate relationship between metrics maximising the left hand-side in~\eqref{ly:1st}, conformal volume, and minimal surfaces. In particular, it is known~\cite{ESI86} that if $(\Sigma^n,g)$ admits an isometric minimal immersion $\phi$ into $S^m$ by first eigenfunctions, then
$$
V_c(m,\Sigma^n)=V_c(m,\phi)=\mathit{Vol}_g(\Sigma^n).
$$
More generally, if $(\Sigma^n,g)$ admits an isometric minimal immersion $\phi$ into $S^m$, then $V_c(m,\Sigma^n)\leqslant\mathit{Vol}_g(\Sigma^n)$. Thus, inequality~\eqref{ly:1st} says that for all metrics $g$ conformal to the metric $g_{\Sigma}$ on a minimal submanifold $\Sigma^n\subset S^m$ we have
$$
\lambda_1(\Sigma^n,g)\mathit{Vol}_g(\Sigma^n)^{2/n}\leqslant n\mathit{Vol}_{g_\Sigma}(\Sigma^n)^{2/n}.
$$
In other words, the volume of a minimal submanifold controls eigenvalues of all conformal metrics on $\Sigma^n$. In particular, setting $g=g_{\Sigma}$ in the inequality above, we obtain $\lambda_1(\Sigma^n,g)\leqslant n$; this is a well-known observation due to Takahashi~\cite{Taka}.

As we show later, this circle of ideas is also closely related to the geometry of submanifolds. In particular, in dimension $2$ Theorem~\ref{ly:theorem} implies the bound on the first Laplace eigenvalue of a submanifold in the Euclidean space $\mathbb R^m$ in terms of the $L^2$-norm of the mean curvature. This bound is a partial case of the following classical inequality established by Reilly~\cite{Re77} in 1977. 
\begin{theorem}
\label{rei1:theorem}
Let $(\Sigma^n,g)$ be a closed Riemannian manifold of dimension $n\geqslant 2$. Then for any isometric immersion $\phi:\Sigma^n\to\mathbb R^m$ the following inequality holds
$$
\lambda_1(\Sigma^n,g)\leqslant\frac{n}{\mathit{Vol}_g(\Sigma^n)}\int_{\Sigma^n}\abs{H_\phi}^2d\mathit{Vol}_g,
$$
where $H_\phi$ is the mean curvature vector of an immersion $\phi$. When $n=m-1$ the equality above occurs if and only if $\Sigma^n$ is a sphere isometrically embedded into $\mathbb R^m$ as a hypersphere. When $n<m-1$ the equality occurs if and only if after scaling the metric $g$ and making a translation and dilation in the ambient space $\mathbb R^m$ the immersion $\phi$ is an isometric minimal immersion into a unit sphere $S^{m-1}\subset\mathbb R^m$ by first eigenfunctions.
\end{theorem}

Mention that similar inequalities hold for isometric immersions into other simply connected constant curvature spaces as well, see~\cite{ESI92}. These results use essentially the fact that these constant curvature spaces admit conformal immersions into a sphere, see details in Section~\ref{proof:rei}.

The purpose of this paper is four-fold: first, we prove a version of Theorem~\ref{ly:theorem} for higher Laplace eigenvalues, showing that the conformal volume actually controls all Laplace eigenvalues. Second, we discuss the relationship with the Reilly inequality and extend the latter to higher Laplace eigenvalues of closed submanifolds in space forms. As an illustration of eigenvalue bounds via conformal volume we also discuss  inequalities for Laplace eigenvalues  on Riemannian surfaces, and prove a bound for the conformal volume of non-orientable surfaces, correcting a version of this statement in~\cite{LiYau82}. Finally, we sharpen known estimates for the number of negative eigenvalues of Schr\"odinger operators, and apply them to obtain index bounds for compact minimal  hypersurfaces $\Sigma^n$ in spheres, which appear to be new. Below we discuss our results in more detail.

\subsection{Conformal volume and higher Laplace eigenvalues}
For a given closed Riemannian manifold $(\Sigma^n,g)$ we denote by
$$
0=\lambda_0(\Sigma^n,g)<\lambda_1(\Sigma^n,g)\leqslant\lambda_2(\Sigma^n,g)\leqslant\ldots\leqslant\lambda_k(\Sigma^n,g)\leqslant\ldots
$$
its Laplace eigenvalues repeated with respect to multiplicity. Our first result is the following version of Theorem~\ref{ly:theorem} for all Laplace eigenvalues.
\begin{theorem}
\label{t:cv}
Let $(\Sigma^n,g)$ be a closed Riemannian manifold of dimension $n\geqslant 2$. Then for any integer $m>0$ such that the $m$-dimensional conformal volume of $\Sigma^n$ is well-defined, for any $k\geqslant 1$ the $k$th Laplace eigenvalue of $\Sigma^n$ satisfies the inequality
\begin{equation}
\label{cvest}
\lambda_k(\Sigma^n,g)\mathit{Vol}_g(\Sigma^n)^{2/n}\leqslant C(n,m)V_c(m,\Sigma^n)^{2/n}k^{2/n},
\end{equation}
where $C(n,m)$ is a constant that depends on the dimensions $n$ and $m$ only.
\end{theorem}
It is important to note that the inequality for Laplace eigenvalues in the theorem above is compatible with the Weyl asymptotic formula:
\begin{equation}
\label{weyl}
\lambda_k(\Sigma^n,g)\mathit{Vol}_g(\Sigma^n)^{2/n}\sim \frac{4\pi^2}{\omega_n^{2/n}}k^{2/n}\qquad\text{as}\quad k\to +\infty,
\end{equation}
where $\omega_n$ is the volume of a unit $n$-dimensional ball in the Euclidean space. Both volume $\mathit{Vol}_g(\Sigma^n)$ and the index $k$ appear in inequality~\eqref{cvest} with the same power as in the Weyl law. As a consequence of Theorem~\ref{t:cv} we obtain the following corollary.
\begin{corollary}
\label{cor:cebm}
Let $(\Sigma^n,g_\Sigma)$ be a minimally immersed closed submanifold of a unit sphere $S^m\subset\mathbb R^{m+1}$. Then for any metric $g$ conformal to $g_\Sigma$ the following inequality holds
\begin{equation}
\label{cb:kth}
\lambda_k(\Sigma^n,g)\mathit{Vol}_g(\Sigma^n)^{2/n}\leqslant C(n,m)\mathit{Vol}_{g_\Sigma}(\Sigma^n)^{2/n}k^{2/n}
\end{equation}
for any $k\geqslant 1$, where $C(n,m)$ is a constant that depends on $n$ and $m$ only.
\end{corollary}
In~\cite{Ko18-} we prove a stronger version of Corollary~\ref{cor:cebm}, where the constant $C(n,m)$ can be chosen to be independent on the ambient dimension $m$. Now we compare Theorem~\ref{t:cv} with other eigenvalue bounds known in the literature. Recall that a celebrated result by Korevaar~\cite{Kor} says that for any closed Riemannian manifold $(\Sigma^n,g)$ its Laplace eigenvalues satisfy the inequalities
\begin{equation}
\label{korevaar}
\lambda_k(\Sigma^n,g)\mathit{Vol}_g(\Sigma^n)^{2/n}\leqslant Ck^{2/n},
\end{equation}
where the constant $C$ depends on the conformal class of a metric $g$ in a rather implicit way. Thus, Theorem~\ref{t:cv} can be viewed as an improvement of Korevaar's result -- it clarifies the way the right hand-side in~\eqref{korevaar} depends on geometry and relates it to a well-known conformal invariant. 

The proof of Theorem~\ref{t:cv} is based on the results originating from the work of Korevaar~\cite{Kor}, and developed further in the papers~\cite{GY99,GNY}; for other developments and applications see~\cite{AH,Ko13,HaKo}. The novelty of our argument is a new construction of test-functions, which relies on specific properties of certain conformal diffeomorphisms of a sphere $S^m$. A version of this construction for complex projective spaces is used in~\cite{Ko18}, where we obtain homological bounds for all Laplace eigenvalues of K\"ahler metrics. Throughout  the paper we assume that $\Sigma^n$ is a closed manifold. However, most of the results continue to hold for compact manifolds with boundary if we impose the Neumann conditions on the boundary.

\subsection{A version of the Reilly inequality for higher eigenvalues}
Now we discuss the Reilly inequality for isometric immersions $\phi:(\Sigma^n,g)\to (\mathbb R^m,g_{\mathit{can}})$, stated in Theorem~\ref{rei1:theorem}. It has been an open question, communicated to us by El Soufi in 2013,  whether there are versions of this inequality for higher Laplace eigenvalues. The only known folkloric result is the following estimate in terms of the $L^\infty$-norm of mean curvature:
$$
\lambda_k(\Sigma^n,g)\leqslant C(n)\max\abs{H_\phi}k^{2/n}
$$
for any $k\geqslant 1$. It is a consequence of the universal inequalities in~\cite[Theorem~2.1]{EHI09} and the recursion formula~\cite[Corollary~2.1]{CY07}. To our knowledge no estimates for higher Laplace eigenvalues $\lambda_k(\Sigma^n,g)$ in terms of the $L^2$-norm of $H_\phi$ exist in the literature.

The analysis of equality cases in Theorems~\ref{ly:theorem} and~\ref{rei1:theorem} suggests that there might be some relationship between the inequalities in these theorems. In dimension $n=2$ this is indeed the case. More precisely, in~\cite[Lemma~1]{LiYau82} Li and Yau show that for surfaces in the Euclidean space $\mathbb R^m$ the following inequality holds
\begin{equation}
\label{vc:mc}
V_c(m,\Sigma^2)\leqslant\int_{\Sigma^2}\abs{H_\phi}^2\mathit{dVol}_g,
\end{equation}
see also~\cite{ESI92}. Combining it with Theorem~\ref{t:cv}, we obtain the following result.
\begin{corollary}
\label{cor:rei}
Let $(\Sigma^2,g_\Sigma)$ be a closed Riemannian surface, and $\phi:(\Sigma^2,g_\Sigma)\to (\mathbb R^m,g_{\mathit{can}})$ be an isometric immersion. Then for any metric $g$ conformal to $g_\Sigma$ the Laplace eigenvalues satisfy the following inequalities
$$
\lambda_k(\Sigma^2,g)\mathit{Vol}_g(\Sigma^2)\leqslant C(m)\left(\int_{\Sigma^2}\abs{H_\phi}^2\mathit{dVol}_{g_\Sigma}\right)k
$$
for any $k\geqslant 1$, where $C(m)$ is a constant that depends on $m$ only.
\end{corollary}
When the dimension $n$ of $\Sigma^n$ is greater than $2$ relation~\eqref{vc:mc} fails to hold. Indeed, in this case the right hand-side in~\eqref{vc:mc} is not invariant under conformal transformations of the ambient space $\mathbb R^m$, and can be made arbitrarily small, while the left hand-side is a genuine conformal invariant. Nevertheless, 
for manifolds of arbitrary dimension we are able to prove the following statement.
\begin{theorem}
\label{t:rei}
Let $(\Sigma^n,g)$ be a closed Riemannian manifold of dimension $n\geqslant 2$, and $\phi:(\Sigma^n,g)\to (\mathbb R^m, g_{\mathit{can}})$ be an isometric immersion. Then for any $k\geqslant 1$ the $k$th Laplace eigenvalue of $\Sigma^n$ satisfies the inequality
$$
\lambda_k(\Sigma^n,g)\leqslant C(n,m)\left(\frac{1}{\mathit{Vol}_g(\Sigma^n)}\int_{\Sigma^n}\abs{H_\phi}^2\mathit{dVol}_{g}\right)k,
$$
where $H_\phi$ is the mean curvature vector, and $C(n,m)$ is a constant that depends on the dimensions $n$ and $m$ only.
\end{theorem}

Theorem~\ref{t:rei} is a consequence of a more general result that we prove in Section~\ref{proof:rei}, and in particular, holds for isometric immersions into constant curvature spaces. Note that the inequality in the theorem above is compatible with Weyl's law~\eqref{weyl} only in dimension $n=2$. It would be interesting to know whether it can be improved to the following statement: under the hypotheses of Theorem~\ref{t:rei}, there exists a constant $C(n)$, depending on the dimension $n$ only, such that
$$
\lambda_k(\Sigma^n,g)\leqslant C(n)\left(\frac{1}{\mathit{Vol}_g(\Sigma^n)}\int_{\Sigma^n}\abs{H_\phi}^2\mathit{dVol}_{g}\right)k^{2/n}
$$
for any $k\geqslant 1$.

\subsection{Weakly conformal immersions, conformal volume, and Yau's problem}
Our next aim is to discuss bounds for the conformal volume, demonstrating some of the applications of Theorems~\ref{ly:theorem} and~\ref{t:cv}. We start with noting that the conformal volume behaves naturally with respect to conformal maps: if $\varphi:\Sigma^n\to M^n$ is a conformal immersion of degree $d$, then
$$
V_c(m,\Sigma^n)\leqslant\abs{d}V_c(m,M^n)
$$
for any $m>0$ such that the conformal volume of $M^n$ is well-defined. Thus, Theorems~\ref{ly:theorem} and~\ref{t:cv} give eigenvalue bounds in terms of degrees of conformal immersions onto model spaces. When $\Sigma^2$ is a closed orientable surface of positive genus, there are no conformal immersions $\Sigma^2\to S^2$, while there are plenty of branched conformal immersions. For this reason in~\cite{LiYau82} Li and Yau consider the notion of the conformal volume defined as the infimum of the conformal volumes of {\em branched conformal immersions}, and show that Theorem~\ref{ly:theorem} continues to hold for it. For our purposes it will be convenient to consider even more general class of maps, described in the following definition.
\begin{defin}
\label{def:wci}
A Lipschitz map $\phi:(\Sigma^n,g)\to (M^m,h)$ is called a {\em weakly conformal immersion} if there exists a closed set $T\subset\Sigma^n$ of zero Lebesgue measure such that the restriction $\left.\phi\right|(\Sigma^n\backslash T)$ is a smooth conformal immersion. The set $T$ is called the {\em singular set} of a weakly conformal immersion $\phi$.
\end{defin}
For a weakly conformal map $\phi:(\Sigma^n,g)\to (S^m,g_{\mathit{can}})$ with a singular set $T$ we define its $m$-dimensional conformal volume by setting
$$
V_c(m,\phi)=\sup\left\{\mathit{Vol}(\Sigma^n\backslash T, (s\circ\phi)^*g_{\mathit{can}})~~ |~ s\text{ is a conformal diffeomorphism of }S^m\right\}.
$$
In Section~\ref{cv:wci} we explain that the volume $\mathit{Vol}(\Sigma^n\backslash T, (s\circ\phi)^*g_{\mathit{can}})$ above can be understood as the integral
$$
n^{-n/2}\int_{\Sigma^n}\abs{\nabla (s\circ\phi)}^n\mathit{dVol}_g=n^{-n/2}\int_{\Sigma^n\backslash T}\abs{\nabla (s\circ\phi)}^n\mathit{dVol}_g,
$$
where the norm $\abs{\nabla (s\circ\phi)}$ is a bounded function by our assumptions. 
\begin{defin}
For a Riemannian manifold $(\Sigma^n,g)$ the infimum of $V_c(m,\phi)$, where $\phi$ ranges over all weakly conformal immersions $(\Sigma^n,g)\to (S^m,g_{\mathit{can}})$, is called the $m$-dimensional {\em weakly conformal volume}, denoted by $V_c^*(m,\Sigma^n)$. 
\end{defin}

Clearly, the definition of the quantity $V_c(m,\phi)$ above is consistent with the definition in~\eqref{cv:def} when $\phi$ is a conformal immersion, and we conclude that $V_c^*(m,\Sigma^n)\leqslant V_c(m,\Sigma^n)$.

In Section~\ref{cv:wci} (see Lemma~\ref{l:degree}) we show that if $\varphi:\Sigma^n\to M^n$ is a weakly conformal immersion with a singular set $T$ such that for any $p\in\varphi(\Sigma^n\backslash T)$ the pre-image $\varphi^{-1}(p)$ has at most $d$ points, then the conformal volumes satisfy the following inequality
\begin{equation}
\label{vc:star}
V_c^*(m,\Sigma^n)\leqslant d V_c^*(m,M^n).
\end{equation} 
In particular, if $\Sigma^2$ is a closed orientable surface of genus $\gamma$, then by the results in~\cite[Chapter~2]{GH}, there exists a holomorphic map $\varphi:\Sigma^2\to S^2$ whose degree is at most $(\gamma+3)/2$. Combining this fact with relation~\eqref{vc:star} and $V_c^*(m,S^2)\leqslant V_c(m,S^2)=4\pi$, we obtain the inequality
\begin{equation}
\label{cv:orient}
V_c^*(m,\Sigma^2)\leqslant 4\pi\left[\frac{\gamma+3}{2}\right]\qquad\text{for any~~}m\geqslant 2,
\end{equation}
where the brackets stand for the integer part. A similar bound for the conformal volume of compact orientable surfaces appears in~\cite{LiYau82}. In the same paper the authors also state a bound in terms of genus for the conformal volume of non-orientable surfaces, but as was pointed out by Karpukhin~\cite{Karp}, their argument is erroneous. Building on the arguments in~\cite{Karp}, we rectify the corresponding statement in~\cite{LiYau82} by proving the following result.
\begin{theorem}
\label{t:cv:nonor}
Let $(\Sigma^2,g)$ be a closed non-orientable Riemannian surface. Then the weakly conformal volume $V^*_c(m,\Sigma^2)$ satisfies the inequality
\begin{equation}
\label{cv:nonor}
V^*_c(m,\Sigma^2)\leqslant 8\pi\left[\frac{\gamma+3}{2}\right]\qquad\text{for any~~}m\geqslant 4,
\end{equation}
where $\gamma$ is the genus of a non-orientable surface, understood as the genus of the orienting double covering.
\end{theorem}
It is important to point out that the inequality for the first Laplace eigenvalue in Theorem~\ref{ly:theorem} continues to hold if we use the weakly conformal volume $V_c^*(m,\Sigma^n)$, that is 
\begin{equation}
\label{ly:1star}
\lambda_1(\Sigma^n,g)\mathit{Vol}_g(\Sigma^n)^{2/n}\leqslant n V_c^*(m,\Sigma^n)^{2/n}.
\end{equation}
Indeed, as we show in Section~\ref{proof:cv} (Lemma~\ref{l:bis}) for a weakly conformal map $\phi:\Sigma^n\to S^m$ the push-forward measure $\phi_*\mathit{Vol}_g$ does not have any atoms and the argument in~\cite{LiYau82,ESI86} carries over directly. In more detail, by Hersch's lemma there exists a conformal diffeomorphism $s:S^m\to S^m$ such that
$$
\int_{S^m}x^id(s\circ\phi)_*\mathit{Vol}_g=0,
$$
and the Lipschitz functions $x^i\circ(s\circ\phi)$ can be used as test-functions for the Rayleigh quotient.

As a consequence, the combination of Theorem~\ref{t:cv:nonor} with inequality~\eqref{ly:1star} yields the following statement, proved by Karpukhin in~\cite{Karp}.
\begin{corollary}
Let $(\Sigma^2,g)$ be a closed non-orientable Riemannian surface. Then its first Laplace eigenvalue satisfies the inequality
$$
\lambda_1(\Sigma^2,g)\mathit{Vol}_g(\Sigma^2)\leqslant 16\pi\left[\frac{\gamma+3}{2}\right],
$$
where $\gamma$ is the genus of a non-orientable surface.
\end{corollary}
In Section~\ref{proof:cv} we explain that Theorem~\ref{t:cv} can be also sharpened to the following statement.
\begin{tcvbis}
Let $(\Sigma^n,g)$ be a closed Riemannian manifold of dimension $n\geqslant 2$. Then for any integer $m>0$ such that the $m$-dimensional weakly conformal volume of $\Sigma^n$ is well-defined, for any $k\geqslant 1$ the $k$th Laplace eigenvalue of $\Sigma^n$ satisfies the inequality
\begin{equation}
\label{cvest:bis}
\lambda_k(\Sigma^n,g)\mathit{Vol}_g(\Sigma^n)^{2/n}\leqslant C(n,m)V_c^*(m,\Sigma^n)^{2/n}k^{2/n},
\end{equation}
where $C(n,m)$ is a constant that depends on the dimensions $n$ and $m$ only.
\end{tcvbis}


In~\cite[Problem~71]{Yau82} Yau has conjectured that for any closed Riemannian surface $(\Sigma^2,g)$ of genus $\gamma$ the normalised Laplace eigenvalues $\lambda_k(\Sigma^2,g)\mathit{Vol}_g(\Sigma^2)$ are bounded above by $C_*(\gamma+1)k$ for some universal constant $C_*$. For orientable Riemannian surfaces this problem was settled by Korevaar in~\cite{Kor}. More precisely, using the existence of branched conformal maps from an orientable surface $\Sigma^2 $ of genus $\gamma$ onto a sphere $S^2$ whose degree is bounded in terms of $\gamma$, the result in~\cite{Kor} yields the following inequality
\begin{equation}
\label{in:kor}
\lambda_k(\Sigma^2,g)\mathit{Vol}_g(\Sigma^2)\leqslant C_1(\gamma+1)k
\end{equation}
for any $k\geqslant 1$, where $C_1$ is a universal constant. The case when $\Sigma^2$ is non-orientable has not been treated in the literature, and it seems that it can not be handled in a similar way, since the existence of branched conformal maps from non-orientable surfaces onto a model $2$-dimensional surface is a more delicate question, see~\cite{Karp}.

Using Theorem~\ref{t:cv}$bis$ together with bounds for the conformal volume we are able to treat the cases when a surface $\Sigma^2$ is orientable or not in a uniform way. More precisely, the combination of Theorem~\ref{t:cv}$bis$ with inequality~\eqref{cv:orient} gives another proof of eigenvalue inequalities~\eqref{in:kor}. Similarly, combining Theorems~\ref{t:cv}$bis$ and~\ref{t:cv:nonor}, we obtain the following eigenvalue bounds for non-orientable surfaces.
\begin{corollary}
\label{in:new}
Let $(\Sigma^2,g)$ be a closed non-orientable Riemannian surface. Then for any $k\geqslant 1$ the $k$th Laplace eigenvalue of $\Sigma^2$ satisfies the inequality
$$
\lambda_k(\Sigma^2,g)\mathit{Vol}_g(\Sigma^2)\leqslant C_2(\gamma+1) k,
$$
where $C_2$ is a universal constant and $\gamma$ is the genus of a non-orientable surface.
\end{corollary}
An estimate similar to the one in Corollary~\ref{in:new} can be also obtained using the method in~\cite{AH} together with an appropriate version of the uniformization theorem for non-orientable Riemannian surfaces. Another application of Theorem~\ref{t:cv:nonor} which does not seem to be treatable by other methods, is an  estimate for the number of negative eigenvalues of Schr\"odinger operators on surfaces. It is a consequence of Theorem~\ref{t:schro} below.

\subsection{Conformal volume and negative eigenvalues of Shr\"odinger operators}
Now we consider the eigenvalue problem for the Schr\"odinger operator
$$
(-\Delta_\Sigma-\mathcal V)u=\lambda u\qquad\text{on~~}\Sigma^n,
$$ 
where $\Delta_\Sigma$ is the Laplace-Beltrami operator on $(\Sigma^n,g)$ and $\mathcal V\in L^p(\Sigma^n)$, where $p>n/2$, is a given potential. As is well-known~\cite{Ma}, under these hypotheses the spectrum of the Schr\"odinger operator $(-\Delta_\Sigma-\mathcal V)$ on a closed manifold $\Sigma^n$ is discrete, and we denote by $N(\mathcal V)$ the number of negative eigenvalues counted with multiplicity. There is an extensive literature on various bounds for $N(\mathcal V)$, see~\cite{Lieb,LiYau83,BiSo, GNY,GN} and references there. In particular, recall the following result due to Grigor'yan, Netrusov, and Yau~\cite{GNY}: for any non-negative potential $\mathcal V\in L^p(\Sigma^n)$, where $p>n/2$, the number of negative eigenvalues satisfies the inequality
\begin{equation}
\label{t:gny}
N(\mathcal V)\geqslant\frac{C}{\mathit{Vol}_g(\Sigma^n)^{n/2-1}}\left(\int_{\Sigma^n}\mathcal Vd\mathit{Vol}_g\right)^{n/2},
\end{equation}
where the constant $C$ depends on the conformal class of a metric $g$ on $\Sigma^n$. When $n=2$ and $\Sigma^2$ is orientable, one can also choose the constant $C$ so that it depends on the genus of $\Sigma^2$ only.

In the spirit of the discussion above, we have the following version of inequality~\eqref{t:gny}, which shows that the dependance of the constant $C$ on the conformal class of $g$ can be incorporated into the conformal volume.
\begin{theorem}
\label{t:schro}
Let $(\Sigma^n,g)$ be a closed Riemannian manifold of dimension $n\geqslant 2$. Then for any integer $m>0$ such that the $m$-dimensional weakly conformal volume of $\Sigma^n$ is well-defined, for any non-negative potential $\mathcal V\in L^p(\Sigma^n)$, $p>n/2$, the number of negative eigenvalues of the Schr\"odinger operator satisfies the inequality
$$
N(\mathcal V)\geqslant\frac{C(n,m)}{V_c^*(m,\Sigma^n)}\frac{1}{\mathit{Vol}_g(\Sigma^n)^{n/2-1}}\left(\int_{\Sigma^n}\mathcal Vd\mathit{Vol}_g\right)^{n/2},
$$
where $C(n,m)$ is a constant that depends on the dimensions $n$ and $m$ only.
\end{theorem}
First, we note that when a potential $\mathcal V$ is bounded, $\mathcal V\in L^{\infty}(\Sigma^n)$, inequality~\eqref{t:gny} continues to hold with the same constant even when $\mathcal V$ is not non-negative, see~\cite{GNS}. Similarly, the argument in~\cite[p.398]{GNS} shows that when $\mathcal V$ is bounded, the inequality in Theorem~\ref{t:schro} holds for not necessarily non-negative potentials.

When the dimension $n$ equals $2$, the upper bounds~\eqref{cv:orient} and~\eqref{cv:nonor} for the conformal volume in terms of genus yield inequalities for the number of negative Schr\"odinger eigenvalues depending on the genus rather than conformal volume for  both orientable and non-orientable surfaces. For orientable surfaces this statement already appears in~\cite{GNY}, but for non-orientable surfaces it is new. 

Another application of Theorem~\ref{t:schro} is concerned with minimal submanifolds in spheres.
\begin{corollary}
\label{t:min}
Let $(\Sigma^n,g)$ be a closed Riemannian manifold of dimension $n\geqslant 2$ that admits an isometric minimal immersion into a unit sphere $S^m$. Then for any non-negative potential $\mathcal V\in L^p(\Sigma^n)$, $p>n/2$, the number of negative eigenvalues of the Schr\"odinger operator satisfies the inequality
$$
N(\mathcal V)\geqslant C(n,m)\left(\frac{1}{\mathit{Vol}_g(\Sigma^n)}\int_{\Sigma^n}\mathcal V\mathit{dVol}_g\right)^{n/2},
$$
where $C(n,m)$ is a constant that depends on the dimensions $n$ and $m$ only.
\end{corollary}
\begin{proof}
By the results of Li and Yau~\cite{LiYau82} and El Soufi and Ilias~\cite{ESI86}, we know that if $\Sigma^n$ admits an isometric minimal immersion into a unit sphere $S^m$, then $V_c(m,\Sigma^n)\leqslant\mathit{Vol}_g(\Sigma^n)$. Since the quantity $V_c^*(m,\Sigma^n)$ is not greater than $V_c(m,\Sigma^n)$, the statement becomes a direct consequence of Theorem~\ref{t:schro}.
\end{proof}
In particular, the statement above yields index estimates for minimally immersed hypersurfaces $\Sigma^n\subset S^{n+1}$. Recall that the index of a two-sided minimal hypersurface $\Sigma^n\subset S^{n+1}$ is defined as the number of negative eigenvalues of the stability operator
$$
J(u)=(-\Delta_\Sigma-n-\abs{S}^2)u,
$$
where $S$ is the shape operator of $\Sigma^n\subset S^{n+1}$ and $u$ is a function on $\Sigma^n$, see~\cite{CM11}. Thus, as a direct consequence of Corollary~\ref{t:min} we see that the index of any closed two-sided minimally immersed hypersurface $\Sigma^n\subset S^{n+1}$ satisfies the inequality
\begin{equation}
\label{index}
\Index (\Sigma^n)\geqslant C(n)\left(n+\frac{1}{\mathit{Vol}(\Sigma^n)}\int_{\Sigma^n}\abs{S}^2\mathit{dVol}_\Sigma\right)^{n/2}.
\end{equation}
To our knowledge this bound for the index is new. However, in dimension $n=2$ the result by Savo~\cite{Savo10} gives a stronger inequality:
$$
\Index (\Sigma^2)\geqslant C_3+C_4\int_{\Sigma^2}\abs{S}^2\mathit{dVol}_\Sigma,
$$
where $C_3$ and $C_4$ are universal constants. As we show in~\cite{Ko18-}, a version of  estimate~\eqref{index}  continues to hold for minimal hypersurfaces in rather general ambient manifolds. In particular, it clarifies some of the qualitative results in~\cite{BS}, which are known to hold for hypersurfaces in dimensions $2\leqslant n\leqslant 6$ only.  As was explained to us by Ben Sharp, it is unlikely that our index bound~\eqref{index} can be improved by replacing the $L_2$-norm of the shape operator by a stronger norm, for example, $L^{n/2}$-norm.

We end this section with the following estimate for the number of negative eigenvalues of Schr\"odinger operators on compact submanifolds in Euclidean spaces.
\begin{theorem}
\label{t:schro2}
Let $(\Sigma^n,g)$ be a closed Riemannian manifold, and $\phi:(\Sigma^n,g)\to (\mathbb R^m,g_{\mathit{can}})$ be an isometric immersion. Then for any non-negative potential $\mathcal V\in L^p(\Sigma^n)$, where $p>n/2$, the number of negative eigenvalues of the Schr\"odinger operator satisfies the inequality
$$
N(\mathcal V)\geqslant C(n,m)\left(\int_{\Sigma^n}\mathcal V\mathit{dVol}_g\right)/\left(\int_{\Sigma^n}\abs{H_\phi}^2\mathit{dVol}_g\right),
$$
where $H_\phi$ is the mean curvature vector, and $C(n,m)$ is a constant that depends on the dimensions $n$ and $m$ only.
\end{theorem}

Theorem~\ref{t:schro2} is a by-product of the circle of ideas around the proof of Theorem~\ref{t:rei}. It is a partial case of a more general statement in Section~\ref{s:schro:proofs}, which in particular applies to submanifolds in constant curvature spaces.

\subsection{Organisation of the paper}
The paper is organised in the following way. In Section~\ref{proof:cv} we summarise the necessary results by Grigoryan, Netrusov, and Yau~\cite{GNY} and describe our construction of test-functions; these two ingredients are used throughout the rest of the paper. In this section we also prove Theorems~\ref{t:cv} and~\ref{t:cv}bis. Section~\ref{proof:rei} is devoted to the proof of the Reilly inequality for higher Laplace eigenvalues (Theorem~\ref{t:rei}) in a rather general setting.  In Section~\ref{cv:wci} we discuss the notion of a weakly conformal map, the related notion of weakly conformal volume, and prove Theorem~\ref{t:cv:nonor}. Finally, in the last section we collect the proofs of bounds for the number of negative eigenvalues of Schrodinger operators.

\smallskip
\noindent
{\em Acknowledgements.} I am grateful to Alessandro Savo and Ben Sharp for discussions on index bounds for minimal hypersurfaces. I am also grateful to Ahmad El Soufi, who passed away in December 2016, for posing a question leading to Theorem~\ref{t:rei}.

\section{Proofs of Theorems~\ref{t:cv} and~\ref{t:cv}bis}
\label{proof:cv}
\subsection{Disjoint charged sets in metric measure spaces}
The proofs of Theorems~\ref{t:cv} and~\ref{t:cv}$bis$, as well as other results in the paper, are based on the existence of a large number of disjoint sets carrying a controlled amount of mass in metric measure spaces. Below we briefly discuss a version of the statement used in the sequel. By $(X, d)$ we denote a separable metric space; the ball $B_r(a)$ in $X$ is a subset of the form $\{x\in X: d(x,a)<r\}$. By an annulus $A$ in a metric space $(X,d)$ we mean a subset of the following form
$$
\{x\in X: r\leqslant d(x,a)<R\},
$$
where $a\in X$ and $0\leqslant r<R<+\infty$. The real numbers $r$ and $R$ above are called the {\em inner} and {\em outer} radii respectively, and the point $a$ is the centre of an annulus $A$. We denote by $2A$ the annulus
$$
\{x\in X:r/2\leqslant d(x,a)<2R\}.
$$
Recall the following definition.
\begin{defin}
For an integer $N> 1$ a metric space $(X,d)$ satisfies the {\em global $N$-covering property}, if each ball $B_r(a)$ can be covered by $N$ balls of radius $r/2$.
\end{defin}
Developing the ideas of Korevaar~\cite{Kor}, Grigoryan and Yau~\cite{GY99} showed that on certain metric spaces with global covering properties for any non-atomic finite measure one can always find a collection of disjoint sets carrying a sufficient amount of measure. We will need the following explicit version of this statement due to Grigoryan, Netrusov, and Yau, see~\cite[Corollary~3.2]{GNY}.
\begin{prop}
\label{ds1}
Let $(X,d)$ be a separable metric space such that all balls $B_r(a)$ are precompact. Suppose that it satisfies the global $N$-covering property for some $N> 1$. Then for any finite non-atomic measure $\mu$ on $(X,d)$ and any positive integer $k$ there exists a collection of $k$ disjoint annuli $\{2A_i\}$ such that
$$
\mu(A_i)\geqslant c\mu(X)/k\qquad\text{for any }1\leqslant i\leqslant k,
$$
where $c$ is a positive constant that depends on $N$ only. (In fact, the constant $c$ can be chosen to be such that $c^{-1}=8N^{12}$.)
\end{prop}
The statement describing the value of the constant $c$ in the proposition follows by examining the main argument in~\cite[Section 3]{GNY}; in particular, see the proof of ~\cite[Lemma~3.4]{GNY}.

\subsection{Construction of test-functions for the Rayleigh quotient}
Let $S^m\subset\mathbb R^{m+1}$ be a unit round sphere. For a given point $p\in S^m$ and a real number $t>0$ we denote by $\xi_{p,t}:S^m\to S^m$ the conformal diffeomorphism 
\begin{equation}
\label{def:xi}
\xi_{p,t}=\phi^{-1}_p\circ s_t\circ\phi_p,
\end{equation}
where $\phi_p:S^{m}\backslash\{p\}\to\mathbb R^m$ is a stereographic projection from a point $p$ to the orthogonal subspace 
$$
L_p=\{x\in\mathbb R^{m+1}: x\cdot p=0\},
$$
 and $s_t:\mathbb R^m\to\mathbb R^m$ is the scaling by $t>0$, that is the map $v\mapsto tv$. Let $x_p:S^m\to\mathbb R$ be a  function $x_p(x)=x\cdot p$, where $x\in S^m$. It is straightforward to see that for any $t>0$ the diffeomorphism $\xi_{p,t}$ fixes the anti-podal points $p$ and $-p$, which are the only extremal points of the function $x_p$. Further, for any $q\in S^m\backslash\{p,-p\}$ the function $x_p$ is strictly increasing along the lines $t\mapsto\xi_{p,t}(q)$, which after a re-parametrisation coincide with the flow lines of $\grad x_p$. 

We start with constructing Lipschitz functions supported in metric balls $B_{2R}(p)\subset S^m$. Our functions are modelled on the restriction of $x_p$ to the hemisphere
$$
S^+_p=\{q\in S^m: q\cdot p>0\},
$$ 
which is a positive smooth function that vanishes on the boundary. For a given $R\in (0,\pi/2)$ we choose the value $t=t(R)>0$ such that $\xi_{p,t}$ maps the ball $B_{2R}(p)$ onto the hemisphere $S^+_p$, and define a function $\varphi_{R,p}$ on $S^m$ by the formula
\begin{equation}
\label{testf1}
\varphi_{R,p}(q)=\left\{
\begin{array}{cc}
x_p(\xi_{p,t}(q)), & \text{ if~~} q\in B_{2R}(p),\\
0, & \text{ if~~} q\notin B_{2R}(p).
\end{array}
\right.
\end{equation}
Clearly, $\varphi_{R,p}$ is a non-negative Lipschitz function on $S^m$; it is supported in the ball $B_{2R}(p)$ and is not greater than $1$ everywhere. For the sequel we will need the following auxiliary lemma.
\begin{lemma}
\label{l1}
For any $R\in (0,\pi/2)$ and any point $p\in S^m$ the function $\varphi_{R,p}$ defined by~\eqref{testf1} satisfies the relation
$$
\varphi_{R,p}(q)\geqslant\frac{3}{5}\qquad\text{for any }q\in B_{R}(p).
$$
\end{lemma}
\begin{proof}
Let $\phi_p$ be a stereographic projection from $p\in S^m$ to the linear subspace orthogonal to $p$, and $t=t(R)>0$ be a real number such that the image $\xi_{p,t}(B_{2R}(p))$ coincides with the upper hemi-sphere $S_p^+$. By a standard argument based on similarity of triangles, it is straightforward to see that the image of the ball $B_{2R}(p)$ under $\phi_p$ is a Euclidean ball in $L_p$ of radius $(\sin2R)/(1-\cos 2R)$, and hence, the value $t=t(R)$ equals
$$
t(R)=\frac{1-\cos 2R}{\sin 2R}=\tan R.
$$
Also, the image of $B_R(p)$ under $\phi_p$ is a Euclidean ball in $L_p$ of radius $(\sin R)/(1-\cos R)$, 
and thus, the image of the ball $\xi_{p,t}(B_R(p))$ under $\phi_p$ is the Euclidean ball in $L_p$ of radius 
$$
\rho(R)=t(R)\frac{\sin R}{1-\cos R}.
$$
After elementary transformations, we obtain
$$
\rho(R)=1+\frac{1}{\cos R}.
$$
Note that $\rho$ as a function of $R\in (0,\pi/2)$ is increasing, and 
$$
\rho(R)\to 2\quad\text{as}\quad R\to 0+\qquad\text{and}\qquad \rho(R)\to+\infty\quad\text{as}\quad R\to\pi/2-.
$$
Thus, the metric ball $\xi_{p,t}(B_R(p))$ is always contained in the ball $\phi_{p}^{-1}(D_2(0))$ in $S^m$, where $D_2(0)$ is a Euclidean ball centred at the origin of radius $2$, and we conclude that the value of the coordinate $x_p$ on $\xi_{p,t}(B_R(p))$ is at least the value of $x_p$ on the boundary $\phi_{p}^{-1}(\partial D_2(0))$. A direct computation shows that this value equals $3/5$, and thus,
$$
\varphi_{R,p}(q)=x_p\circ\xi_{p,t}(q)\geqslant\frac{3}{5}\quad\text{for any}\quad q\in B_R(p), 
$$
where $R\in (0,\pi/2)$.
 \end{proof}

In a similar fashion, for a given $r\in (0,\pi)$ we choose the value $\tau>0$ such that the set $\xi_{p,\tau}(B_{r/2}(p))$ coincides with the upper hemisphere $S_p^+$, and define a function $\bar\varphi_{r,p}$ by the formula
\begin{equation}
\label{testf2}
\bar\varphi_{r,p}(q)=\left\{
\begin{array}{cc}
0, & \text{ if~~} q\in B_{r/2}(p),\\
-x_p(\xi_{p,\tau}(q)), & \text{ if~~} q\notin B_{r/2}(p).
\end{array}
\right.
\end{equation}
It is a non-negative Lipschitz function, which is supported in the complement of the ball $B_{r/2}(p)$ and is not greater than $1$ everywhere. The proof of the following statement is similar to Lemma~\ref{l1}, and therefore, is omitted. 
\begin{lemma}
\label{l2}
For any $r\in (0,\pi)$ and any point $p\in S^m$ the function $\bar\varphi_{r,p}$ defined by~\eqref{testf2} satisfies the relation
$$
\bar\varphi_{r,p}(q)\geqslant\frac{3}{5}\qquad\text{for any }q\notin B_{r}(p).
$$
\end{lemma}
Now let $A\subset S^m$ be an annulus $B_{R}(p)\backslash B_r(p)$, where $0\leqslant r<R<\pi/2$ and $p\in S^m$, and $2A$ be an annulus $B_{2R}(p)\backslash B_{r/2}(p)$. We define a function $u_A$ on $S^m$ by setting it to be the product $\varphi_{R,p}\bar\varphi_{r,p}$. Clearly, it is a Lipschitz function that is supported in $2A$ and satisfies the  relations $0\leqslant u_A\leqslant 1$ and 
$$
u_A(q)\geqslant\frac{9}{25}\qquad\text{for any }q\in A.
$$
We use the pull-backs of such functions as test-functions for the Rayleigh quotient to complete the proof of Theorem~\ref{t:cv} below.

\subsection{Proof of Theorem~\ref{t:cv}: final argument}
Recall that the Rayleigh quotient $\mathcal R(u)$ on $\Sigma^n$ is defined as
$$
\mathcal R(u)=\left(\int_{\Sigma^n}\abs{\nabla u}^2d\mathit{Vol}_g\right)/\left(\int_{\Sigma^n}u^2d\mathit{Vol}_g\right),
$$
where $u$ is an admissible test-function. By the variational principle, see~\cite{Cha}, for a proof of the theorem it is sufficient for any $k\geqslant 1$ and any conformal immersion $\phi:\Sigma^n\to S^m$ to construct $k+1$ linearly independent Lipschitz test-functions  $u_i$ such that for any $u\in\mathit{Span}(u_i)$ the following inequality holds
\begin{equation}
\label{t:cv:punch}
\mathcal R(u)\leqslant C(n,m)\mathit{Vol}(\Sigma^n,g)^{-2/n} V_c(m,\phi)^{2/n}k^{2/n},
\end{equation}
where the quantity $V_c(m,\phi)$ is defined in~\eqref{cv:def}.

We view a unit sphere $S^m\subset\mathbb R^{m+1}$ as a metric space with the intrinsic distance function $d$, that is the distance between two given points on it is the infimum of lengths of piece-wise smooth paths in $S^m$ joining these points. Using standard formulae for the volumes of metric balls on $S^m$, for example in~\cite{Cha}, it is straightforward to see that the metric space $(S^m,d)$ satisfies the global $N$-covering property with $N=9^m$. We endow $(S^m,d)$ with a measure $\mu$ obtained as the push-forward of the volume measure $\mathit{Vol}_g$ on $\Sigma^n$ under a given conformal immersion $\phi:\Sigma^n\to S^m$. Since $\Sigma^n$ is closed, it is straightforward to see that the pre-image $\phi^{-1}(p)$ of any point $p\in S^m$ is either empty or a finite set, and hence, the measure $\mu$ is non-atomic. Thus, Proposition~\ref{ds1} applies and we can find a collection $\{A_i\}$ of $2(k+1)$ annuli on the sphere such that
\begin{equation}
\label{denom1}
\mu(A_i)\geqslant c\mu(S^m)/(2k+2)\geqslant c\mu(S^m)/(4k)\qquad\text{for any }i=1,\ldots,2k+2,
\end{equation}
and the annuli $\{2A_i\}$ are disjoint. The last property implies that 
$$
\sum_{i=1}^{2k+2}\mu(2A_i)\leqslant\mu (S^m),
$$
and hence, there exists at least $k+1$ sets $2A_i$ such that
\begin{equation}
\label{denom2}
\mu(2A_i)\leqslant \mu(S^m)/(k+1)\leqslant \mu(S^m)/k.
\end{equation}
Without loss of generality, we may assume that these inequalities hold for $i=1,\ldots,k+1$. We denote by $u_i$ the Lipschitz test-functions $u_{A_i}\circ\phi$, where $u_{A_i}$ are constructed above. In more detail, let $p_i$, $r_i$, and $R_i$ be the centre, the inner radius and the outer radius of $A_i$ respectively. Denote by $\varphi_i$ the functions $\varphi_{R_i,p_i}$, and by $\bar\varphi_i$ the function $\bar\varphi_{r_i,p_i}$, see the construction above. Then the function 
$$
u_i=\left\{
\begin{array}{cc}
(\varphi_i\bar\varphi_i)\circ\phi, &\text{ if~~ }r_i>0,\\
\varphi_i\circ\phi, & \text{ if~~ }r_i=0,
\end{array}
\right.
$$ 
can be used as a test-function for the Rayleight quotient on $\Sigma^n$. Since the $u_i$'s are supported in the disjoint sets $\phi^{-1}(2A_i)$, they are $W^{1,2}$-orthogonal, and it is sufficient to prove inequality~\eqref{t:cv:punch} for all $u_i$, where $i=1,\ldots,k+1$. 

To prove inequality~\eqref{t:cv:punch} for each $u_i$, we first estimate the numerator in the Rayleigh quotient. Below we assume that $r_i>0$; the case $r_i=0$ can be treated similarly. Since the functions $\varphi_i$ and $\bar\varphi_i$ are not greater than $1$, we obtain
\begin{multline*}
\int_{\Sigma^n}\abs{\nabla u_i}^2d\mathit{Vol}_g\leqslant 2\left(\int_{\phi^{-1}(2A_i)}\abs{\nabla(\varphi_i\circ\phi)}^2d\mathit{Vol}_g+\int_{\phi^{-1}(2A_i)}\abs{\nabla(\bar\varphi_i\circ\phi)}^2d\mathit{Vol}_g\right)\\ \leqslant 2\left(\left(\int_{\Sigma^n}\abs{\nabla(\varphi_i\circ\phi)}^nd\mathit{Vol}_g\right)^{2/n}+\left(\int_{\Sigma^n}\abs{\nabla(\bar\varphi_i\circ\phi)}^nd\mathit{Vol}_g\right)^{2/n}\right)\left(\int_{\phi^{-1}(2A_i)} 1d\mathit{Vol}_g\right)^{1-2/n},
\end{multline*}
where in the last relation we used the H\"older inequality. Using the definition of the function $\varphi_i$, we can estimate the first integral above in the following way
\begin{multline*}
\int_{\Sigma^n}\abs{\nabla(\varphi_i\circ\phi)}^nd\mathit{Vol}_g\leqslant \int_{\Sigma^n}\abs{\nabla(x_{p_i}\circ(\xi_{p_i,t_i}\circ\phi))}^nd\mathit{Vol}_g\leqslant\int_{\Sigma^n}\abs{\nabla(\xi_{p_i,t_i}\circ\phi)}^nd\mathit{Vol}_g\\ =n^{n/2}\mathit{Vol}(\Sigma^n,(\xi_{p_i,t_i}\circ\phi)^*g_{\mathit{can}})\leqslant n^{n/2}V_c(m,\phi),
\end{multline*}
where $\xi_{p_i,t_i}$ is a conformal transformation of $S^m$, $g_{\mathit{can}}$ is a canonical round metric on $S^m$, and in the equality above we used the relation
$$
(\xi_{p_i,t_i}\circ\phi))^*g_{\mathit{can}}=\frac{1}{n}\abs{\nabla(\xi_{p_i,t_i}\circ\phi)}^2g.
$$
Similarly, we obtain
$$
\int_{\Sigma^n}\abs{\nabla(\bar\varphi_i\circ\phi)}^nd\mathit{Vol}_g\leqslant n^{n/2}V_c(m,\phi).
$$
Combining these relations  with the estimate for the integral $\int_{\Sigma^n}\abs{\nabla u_i}^2d\mathit{Vol}_g$ above, we further obtain
\begin{multline}
\label{num}
\int_{\Sigma^n}\abs{\nabla u_i}^2d\mathit{Vol}_g\leqslant 4nV_c(m,\phi)^{2/n}\mu(2A_i)^{1-2/n}\leqslant 4nV_c(m,\phi)^{2/n}(\mu(S^m)/k)^{1-2/n}\\ =4nV_c(m,\phi)^{2/n}(\mathit{Vol}_g(\Sigma^n)/k)^{1-2/n},
\end{multline}
where in the second inequality we used relation~\eqref{denom2}. Using Lemmas~\ref{l1} and~\ref{l2} together with relation~\eqref{denom1}, we can also estimate the denominator of the Rayleigh quotient:
\begin{equation}
\label{denom3}
\int_{\Sigma^n}u_i^2d\mathit{Vol}_g\geqslant \frac{81}{625}\mathit{Vol}_g(\phi^{-1}(A_i))=\frac{81}{625}\mu(A_i)\geqslant\frac{81}{2500}c\mathit{Vol}_g(\Sigma^n)/k,
\end{equation}
where the constant $c$ depends only on $m$. Relations~\eqref{num} and~\eqref{denom3} now immediately imply inequality~\eqref{t:cv:punch} for all $u_i$, where $i=1,\ldots,k+1$.
\qed

\subsection{Proof of Theorem~\ref{t:cv}bis}
We start with the following auxiliary statement. Its proof is rather starightforward, however, we state it as a lemma for the convenience of references.
\begin{lemma}
\label{l:bis}
Let $(\Sigma^n,g)$ be a closed Riemannian manifold, and $\phi:\Sigma^n\to M^m$ be a weakly conformal immersion in the sense of Definition~\ref{def:wci}. Then the push-forward measure $\mu=\phi_*\mathit{Vol}_g$ is non-atomic, that is $\mu(p)=0$ for any point $p\in S^m$.
\end{lemma}
\begin{proof}
Suppose the contrary, and let $p\in M^m$ be a point such that $\mu(p)>0$, that is, the pre-image $\phi^{-1}(p)$ has positive measure $\mathit{Vol}_g$. Since the singular set $T$ of $\phi$ has zero measure, we conclude that the value $\mathit{Vol}_g(\phi^{-1}(p)\backslash T)$ is positive. Since $\phi$ is a smooth immersion on $\Sigma^n\backslash T$, the constant rank theorem implies that it is locally injective, and hence, the set $\phi^{-1}(p)\backslash T$ is either finite or consists of a countable set of points that accumulate to $T$. Thus, its measure equals zero, and we arrive at a contradiction.
\end{proof}

The lemma above guarantees that Proposition~\ref{ds1} applies to the metric space $(S^m,d)$ equipped with the measure $\mu$, and the argument in the proof of Theorem~\ref{t:cv} carries over. In more detail, by Proposition~\ref{ds1} we can find a collection $\{A_i\}$ of $k+1$ annuli such that the annuli $\{2A_i\}$ are disjoint and inequalities~\eqref{denom1} and~\eqref{denom2} hold. Since $\phi:\Sigma^n\to S^m$ is a Lipschitz map, the test-functions $u_i$ constructed in the proof of Theorem~\ref{t:cv} are also Lipschitz. In particular, we see that
$$
\int_{\Sigma^n}\abs{\nabla u_i}^2\mathit{dVol}_g=\int_{\Sigma^n\backslash T}\abs{\nabla u_i}^2\mathit{dVol}_g
$$
where $T$ is a zero measure singular set of $\phi$. Thus, the numerator of the Rayleigh quotient of $u_i$ is determined by its values on $\Sigma^n\backslash T$ only, and since $\phi$ is a smooth conformal immersion on this set, we see that all estimates for the Dirichlet energy of $u_i$ in the proof of Theorem~\ref{t:cv} carry over. So do the estimates for the denominator of the Rayleigh quotient $\mathcal R(u_i)$.
\qed

\section{Reilly inequalities for higher eigenvalues}
\label{proof:rei}

\subsection{Submanifolds in the spaces that admit conformal immersions to a sphere}
The purpose of this section is to discuss a circle of ideas related to the proof of Theorem~\ref{t:rei}. In fact, we prove a version of the theorem in a much more general setting. Before stating it, we introduce the following notation. Given a Riemannian manifold $(M^m,h)$ and an isometric immersion $\phi:\Sigma^n\to M^m$ by $R_\phi$ we denote the quantity
\begin{equation}
\label{n:rphi}
R_\phi(x)=\frac{1}{n(n-1)}\sum_{i\ne j}K_M(d\phi(e_i),d\phi(e_j)),
\end{equation}
where the symbol $K_M$ denotes the sectional curvature of $(M^m,h)$, and $(e_i)$ is an orthonormal basis of $T_x\Sigma^n$ at a point under consideration. In particular, if $(M^m,h)$ is a space of constant curvature $\kappa$, then for any isometric immersion $\phi:\Sigma^n\to M^m$ the function $R_\phi$ is constant, $R_\phi\equiv\kappa$.

Theorem~\ref{t:rei} is a consequence of the following more general result.
\begin{theorem}
\label{t:grei}
Let $(M^m,h)$ be a (possibly non-complete) Riemannian manifold that admits a conformal immersion into a unit sphere $S^m\subset\mathbb R^{m+1}$ of the same dimension $m\geqslant 2$. Then for any closed Riemannian manifold $(\Sigma^n,g)$ of dimension $n\geqslant 2$ and any isometric immersion $\phi:\Sigma^n\to M^m$ the Laplace eigenvalues of $\Sigma^n$ satisfy inequalities
$$
\lambda_k(\Sigma^n,g)\leqslant C(n,m)\left(\frac{1}{\mathit{Vol}_g(\Sigma^n)}\int_{\Sigma^n}\left(\abs{H_\phi}^2+R_\phi\right)\mathit{dVol}_{g}\right)k
$$
for any $k\geqslant 1$, where $H_\phi$ is the mean curvature vector of $\phi$, and $R_\phi$ is given by relation~\eqref{n:rphi}.
\end{theorem}

The theorem above implies the following version for higher Laplace eigenvalues of the classical Reilly inequalities~\cite{Re77,ESI92} for the first eigenvalue of submanifolds in constant curvature spaces.
\begin{corollary}
Let $(M^m,h)$ be a complete simply connected space of constant curvature $\kappa\in\mathbb R$. Then for any closed Riemannian manifold $(\Sigma^n,g)$ of dimension $n\geqslant 2$ and any isometric immersion $\phi:\Sigma^n\to M^m$ the Laplace eigenvalues of $\Sigma^n$ satisfy the inequalities
$$
\lambda_k(\Sigma^n,g)\leqslant C(n,m)\left(\frac{1}{\mathit{Vol}_g(\Sigma^n)}\int_{\Sigma^n}\left(\abs{H_\phi}^2+\kappa\right)\mathit{dVol}_{g}\right)k
$$
for any $k\geqslant 1$, where $H_\phi$ is the mean curvature vector of $\phi$.
\end{corollary}

One of the key ingredients in the proof of Theorem~\ref{t:grei} is the following relation, found by El Soufi and Ilias~\cite[Prop.~2]{ESI92}.
\begin{prop}
\label{f:esi}
Let $(M^m,h)$ be a (possibly non-complete) Riemannian manifold of dimension $m\geqslant 2$, and $\Pi:M^m\to S^m\subset\mathbb R^{m+1}$ be a conformal immersion, $\Pi^*g_{\mathit{can}}=e^fh$. Then for any immersion $\phi:\Sigma^n\to M^m$ the following relation holds:
$$
\abs{H_\phi}^2+R_\phi=e^{f\circ\phi}\left(\abs{H_{\Pi\circ\phi}}^2+1\right)+\frac{n-2}{4n}\abs{\nabla (f\circ\phi)}^2-\frac{1}{n}\Delta(f\circ\phi),
$$
where the norm $\abs{\nabla (f\circ\phi)}$ and the Laplacian $\Delta(f\circ\phi)$ are taken with respect to the metric $\phi^*h$.
\end{prop}
The proof of Prop.~\ref{f:esi} uses standard formulae for the behaviour of the geometric quantities under a conformal change of the metric together with the Gauss equations, see details in~\cite{ESI92}. Recall that the energy density of a map between Riemannian manifolds is defined as the Hilbert-Schmidt norm of its differential. In particular, for the conformal map $\Pi\circ\phi:(\Sigma^n,\phi^*h)\to (S^m,g_{\mathit{can}})$ it satisfies the relation
$$
\abs{\nabla (\Pi\circ\phi)}^2:=\sum_{i=1}^{m+1}\abs{\nabla (\Pi\circ\phi)^i}^2=ne^{f\circ\phi},
$$
where we view $S^m$ as a unit sphere in $\mathbb R^{m+1}$, and use the notation from Prop.~\ref{f:esi}, that is $\Pi^*g_{\mathit{can}}=e^fh$. Thus, as a consequence of Prop.~\ref{f:esi}, we arrive at the following statement.
\begin{corollary}
\label{c:rei}
Let $(M^m,h)$ be a (possibly non-complete) Riemannian manifold that admits a conformal immersion into a unit sphere $S^m\subset\mathbb R^{m+1}$ of the same dimension $m\geqslant 2$. Then for any closed Riemannian manifold $(\Sigma^n,g)$ of dimension $n\geqslant 2$ and any isometric immersion $\phi:\Sigma^n\to M^m$, the following inequality holds:
\begin{equation}
\label{f2:esi}
\int\limits_{\Sigma^n}\left(\abs{H_\phi}^2+R_\phi\right)\mathit{dVol}_g\geqslant\frac{1}{n}\sup_\Pi\int\limits_{\Sigma^n}\abs{\nabla (\Pi\circ\phi)}^2\mathit{dVol}_g,
\end{equation}
where the supremum is taken over all conformal immersions $\Pi:(M^m,h)\to (S^m,g_{\mathit{can}})$.
\end{corollary}

\subsection{Proof of Theorem~\ref{t:grei}}
We follow the line of argument in the proof of Theorem~\ref{t:cv}. Let $\phi:(\Sigma^n,g)\to (M^m,h)$ be a given isometric immersion. For a proof of the theorem it is sufficient for any integer $k\geqslant 1$ to construct $k+1$ linearly independent Lipschitz test-functions $u_i$ such that for any $u\in\mathit{Span}(u_i)$ the following inequality holds
\begin{equation}
\label{aux:ray}
\mathcal R(u)\leqslant C(n,m)\left(\frac{1}{\mathit{Vol}_g(\Sigma^n)}\int_{\Sigma^n}\left(\abs{H_\phi}^2+R_\phi\right)\mathit{dVol}_{g_\Sigma}\right)k,
\end{equation}
where $\mathcal R(u)$ is the Rayleigh quotient on $\Sigma^n$.

Let $\Pi:(M^m,h)\to (S^m,g_{\mathit{can}})$ be a conformal immersion, and denote by $\psi$ the composition $\Pi\circ\phi:\Sigma^n\to S^m$. We view the unit sphere $S^m\subset\mathbb R^{m+1}$ as a metric space with the intrinsic distance function $d$. By $\mu$ we denote a measure on $(S^m,d)$ obtained by pushing forward the volume measure $\mathit{Vol}_g$ on $\Sigma^n$ under the conformal immersion $\psi:\Sigma^n\to S^m$. By Proposition~\ref{ds1} we find a collection $\{A_i\}$ of $k+1$ annuli on $(S^m,d)$ such that the annuli $\{2A_i\}$ are disjoint, and
\begin{equation}
\label{denom:rei}
\mu(A_i)\geqslant c\mu(S^m)/(k+1)\geqslant c\mu(S^m)/(2k),
\end{equation}
for any $i=1,\ldots,k+1$, where the constant $c$ depends on the dimension $m$ only. By $u_i$ we denote the Lipschitz test-functions $u_{A_i}\circ\psi$, constructed in Section~\ref{proof:cv}. Since the functions $u_i$ are supported in the disjoint sets $\psi^{-1}(2A_i)$, they are $W^{1,2}$-orthogonal, and it is sufficient to prove inequality~\eqref{aux:ray} for each $u_i$. 

First, we estimate the numerator in the Rayleigh quotient:
\begin{equation}
\label{rei:n}
\int_{\Sigma^n}\abs{\nabla u_i}^2d\mathit{Vol}_g\leqslant 2\left(\int_{\Sigma^n}\abs{\nabla(\varphi_i\circ\psi)}^2d\mathit{Vol}_g+\int_{\Sigma^n}\abs{\nabla(\bar\varphi_i\circ\psi)}^2d\mathit{Vol}_g\right),
\end{equation}
where the functions $\varphi_i=\varphi_{R_i,p_i}$, and $\bar\varphi_i=\bar\varphi_{R_i,p_i}$ are given by formulae~\eqref{testf1} and~\eqref{testf2} respectively. Using the definition of the function $\varphi_i$, we can bound the first integral in the following way
\begin{multline*}
\int_{\Sigma^n}\abs{\nabla(\varphi_i\circ\psi)}^2d\mathit{Vol}_g\leqslant \int_{\Sigma^n}\abs{\nabla(x_{p_i}\circ(\xi_{p_i,t_i}\circ\psi))}^2d\mathit{Vol}_g\leqslant\int_{\Sigma^n}\abs{\nabla(\xi_{p_i,t_i}\circ\Pi\circ\phi)}^2d\mathit{Vol}_g\\ \leqslant n\int_{\Sigma^n}\left(\abs{H_\phi}^2+R_\phi\right)\mathit{dVol}_g,
\end{multline*}
where $\xi_{p_i,t_i}$ is a conformal transformation of $S^m$, and in the last inequality we used Corollary~\ref{c:rei}. In a similar way, we can estimate the second term in~\eqref{rei:n}, and summing up, obtain
\begin{equation}
\label{rei:n2}
\int_{\Sigma^n}\abs{\nabla u_i}^2\mathit{dVol}_g\leqslant 4n\int_{\Sigma^n}\left(\abs{H_\phi}^2+R_\phi\right)\mathit{dVol}_g.
\end{equation}
Now, using relation~\eqref{denom:rei}, we estimate the denominator in the Rayleigh quotient:
\begin{equation}
\label{rei:den}
\int_{\Sigma^n}u_i^2d\mathit{Vol}_g\geqslant \frac{81}{625}\mathit{Vol}_g(\psi^{-1}(A_i))=\frac{81}{625}\mu(A_i)\geqslant\frac{81}{1250}c\mathit{Vol}_g(\Sigma^n)/k,
\end{equation}
where the constant $c$ depends only on $m$. The combination of inequalities~\eqref{rei:n2} and~\eqref{rei:den} yields estimate~\eqref{aux:ray} for the Rayleigh quotient $\mathcal R(u)$ for all functions $u_i$,  where $i=1,\ldots, k+1$, and proves the theorem.
\qed

\section{Conformal volume and weakly conformal immersions}
\label{cv:wci}
\subsection{Preliminary considerations}
We start with a few remarks on the definition of a weakly conformal immersion, see Definition~\ref{def:wci}. A useful model example of such a map that is not smooth is the map $\psi_p:S^m\to S^m$ of the unit sphere $S^m\subset\mathbb R^{m+1}$ to itself defined as
$$
\psi_p(x)=\left\{
\begin{array}{cc}
x, & \text{ if }x\cdot p\geqslant 0;\\
r_px, & \text{ if }x\cdot p\leqslant 0.
\end{array}
\right.
$$
Here $p\in S^m$ is a fixed point, and $r_p:S^m\to S^m$ is the reflection that fixes the orthogonal complement $L_p=\{x: x\cdot p=0\}$. A similar map is used in the proof of Theorem~\ref{t:cv:nonor}.

Recall that for a weakly conformal map $\phi:(\Sigma^n,g)\to (S^m,g_{\mathit{can}})$ we define its $m$-dimensional conformal volume by the formula
$$
V_c(m,\phi)=\sup\left\{\mathit{Vol}(\Sigma^n\backslash T, (s\circ\phi)^*g_{\mathit{can}})~~ |~ s\text{ is a conformal diffeomorphism of }S^m\right\},
$$
where $T$ is the singular set of $\phi$. Note that for a smooth conformal immersion $\psi:\Sigma^n\backslash T\to S^m$ the relation
$$
\psi^*g_{\mathit{can}}=\frac{1}{n}\abs{\nabla\psi}^2g
$$
holds, and hence, the volume $\mathit{Vol}(\Sigma^n\backslash T, (s\circ\phi)^*g_{\mathit{can}})$ is given by the integral
$$
n^{-n/2}\int_{\Sigma^n\backslash T}\abs{\nabla (s\circ\phi)}^n\mathit{dVol}_g=n^{-n/2}\int_{\Sigma^n}\abs{\nabla (s\circ\phi)}^n\mathit{dVol}_g,
$$
where in the last relation we used the fact that the norm $\abs{\nabla(s\circ\phi)}$ is a bounded function on $\Sigma^n\backslash T$.

The following statement is a basis for our estimates for the conformal volume.
\begin{lemma}
\label{l:degree}
Let $(\Sigma^n,g)$ and $(M^n,h)$ be two Riemannian manifolds of the same dimension, and $\varphi:\Sigma^n\to M^n$ be a weakly conformal immersion with a singular set $T$. Suppose that for any $p\in\varphi(\Sigma^n\backslash T)$ the pre-image $\varphi^{-1}(p)$ has at most $d$ points. Then for any $m>0$ such that there exists a weakly conformal immersion of $M^n$ to a unit sphere $S^m$, the following relation holds:
$$
V_c^*(m,\Sigma^n)\leqslant d V_c^*(m,M^n).
$$
\end{lemma}
\begin{proof}
Let $\phi:M^n\to S^m$ be a weakly conformal immersion with a singular set $R$. Then the composition $\phi\circ\varphi:\Sigma^n\to S^m$ is a weakly conformal immersion whose singular set is contained in $T\cup\varphi^{-1}(R)$. Note that the pre-image $\varphi^{-1}(R)$ has zero Lebesgue measure. In more detail, the pre-image $\varphi^{-1}(R)$ is the union of $\varphi^{-1}(R)\cap T$ and $\varphi^{-1}(R\cap\varphi(\Sigma^n\backslash T))$, and the former set clearly has zero measure. Using the hypotheses of the lemma together with the constant rank theorem, it is straightforward to see that the restriction $\left.\varphi\right|\Sigma^n\backslash T$ is a covering map over every connected component of its image, and hence, the pre-image of any zero measure subset in $\varphi(\Sigma^n\backslash T)$ is a zero measure subset. Thus, we conclude that the set $\varphi^{-1}(R\cap\varphi(\Sigma^n\backslash T))$ has zero measure as well.

Denote by $\Omega$ the complement $\Sigma^n\backslash (T\cup\varphi^{-1}(R))$. The discussion above implies that the relation
$$
V_c(m,\phi\circ\varphi)=V_c(m,\left.\phi\circ\varphi\right|\Omega)
$$
holds. Further, since for any $p\in\varphi(\Omega)$ the pre-image $\varphi^{-1}(p)$ is finite, and $\left.\varphi\right|\Omega$ is an immersion, we conclude that that $\left.\varphi\right|\Omega$ is a finite covering map over every connected component of $\varphi(\Omega)$. Using this it is straightforward to show that
$$
\mathit{Vol}(\Omega,\varphi^*h)\leqslant d\mathit{Vol}(\varphi(\Omega),h)\leqslant d\mathit{Vol}(M^n\backslash R,h)
$$
for an arbitrary metric $h$ on $M^n$. In particular, for any conformal diffeomorphism $s:S^m\to S^m$, we may set $h=(s\circ\phi)^*g_{\mathit{can}}$ to obtain
$$
\mathit{Vol}(\Omega,(s\circ\phi\circ\varphi)^*g_{\mathit{can}})\leqslant d\mathit{Vol}(M^n\backslash R,(s\circ\phi)^*g_{\mathit{can}}).
$$
The latter clearly implies the following relations between conformal volumes:
$$
V_c^*(m,\Sigma^n)\leqslant V_c(m,\phi\circ\varphi)=V_c(m,\left.\phi\circ\varphi\right|\Omega)\leqslant d V_c(m,\phi).
$$
Now taking the infimum over all weakly conformal immersions $\phi:M^n\to S^m$, we obtain the inequality in the statement.
\end{proof}

\subsection{Proof of Theorem~\ref{t:cv:nonor}}
We start with introducing some notation. Let $\Sigma^2$ be a closed non-orientable surface, and $\tilde\Sigma^2$ be its orienting double covering. The group of deck transformations is generated by an involution $\sigma:\tilde\Sigma^2\to\tilde\Sigma^2$, which is an isometry with respect to the the pull-back metric $\pi^*g$, where $\pi:\tilde\Sigma^2\to\Sigma^2$ is a covering map. Let $\tau:S^2\to S^2$ be a conformal diffeomorphism that changes an orientation on $S^2$ and is an involution, that is $\tau^2=\mathit{id}$. Throughout the rest of the section by a $\tau$-{\em equivariant conformal map} $u:\tilde\Sigma^2\to S^2$ we mean a {\em branched conformal immersion} that satisfies the relation
$$
u(\sigma(x))=\tau(u(x))\qquad\text{for any }\quad x\in\tilde\Sigma^2.
$$
Denote by $F$ the fixed point set of an involution $\tau$ on $S^2$. As follows from the definition, for any $\tau$-equivariant conformal map $u:\tilde\Sigma^2\to S^2$ the involution $\sigma$ preserves the set $u^{-1}(F)$ as well as the set of branch points of $u$. Viewing the sphere $S^2$ as the extended complex plane, we denote by $\tau_1:S^2\to S^2$ the reflection $z\mapsto\bar z$, and by $\tau_2:S^2\to S^2$ the anti-podal map $z\mapsto -\bar z^{-1}$.

In~\cite{Karp} Karpukhin proves the following statement.
\begin{prop}
\label{p:karp}
Let $(\Sigma^2,g)$ be a closed non-orientable Riemannian surface, and $\tilde\Sigma^2$ be its orienting double covering. Then there exists either a $\tau_1$-equivariant conformal map $\tilde\Sigma^2\to S^2$ whose degree is not greater than $2[(\gamma+3)/2]$ or a $\tau_2$-equivariant conformal map $\tilde\Sigma^2\to S^2$ whose degree is not greater than $[(\gamma+3)/2]$.
\end{prop}
Now we finish the proof of Theorem~\ref{t:cv:nonor}. Suppose that the first case in Prop.~\ref{p:karp} occurs: there exists a $\tau_1$-equivariant conformal map $u$ whose degree $d$ is not greater than $2[(\gamma+3)/2]$. Since $\tau_1$ is a reflection that fixes an equator $F$, such a map $u$ defines a weakly conformal map $\varphi:\Sigma^2\to S^2_+\subset S^2$, where $S^2_+$ is an upper hemisphere in $S^2$ such that $\partial S^2_+=F$. In more detail, let $B\subset\tilde\Sigma^2$ be a collection of branch points of $u$. Then the induced map $\varphi:\Sigma^2\to S^2$ fails to be a smooth conformal immersion precisely on the set
$$
T=\{x\in\Sigma^2: \pi^{-1}(x)\in B\text{ or } u(\pi^{-1}(x))\in F\}.
$$
Since $u:\tilde\Sigma^2\to S^2$ is a branched immersion, it is straightforward to conclude that the set $T$ has zero Lebesgue measure. Thus, Lemma~\ref{l:degree} applies to the map $\varphi$, and we conclude that
$$
V_c^*(m,\Sigma^2)\leqslant 2\left[\frac{\gamma+3}{2}\right]V_c^*(m,S^2)
$$
for any $m\geqslant 2$, where the brackets stand for the integer part. By the results in~\cite{LiYau82}, we know that
$$
V_c^*(m,S^2)\leqslant V_c(m,S^2)=4\pi,
$$
and hence, we finally obtain
\begin{equation}
\label{p:case1}
V_c^*(m,\Sigma^2)\leqslant 8\pi\left[\frac{\gamma+3}{2}\right]\qquad\text{for any }\quad m\geqslant 2.
\end{equation}
Now we consider the second case in Prop.~\ref{p:karp}: there exists a $\tau_2$-equivariant conformal map $u$ whose degree $d$ is not greater than $[(\gamma+3)/2]$. Then such a  map $u$ defines a branched conformal immersion $\varphi:\Sigma^2\to\mathbb RP^2$, and by Lemma~\ref{l:degree}, we again obtain
$$
V_c^*(m,\Sigma^2)\leqslant \left[\frac{\gamma+3}{2}\right]V_c^*(m,\mathbb RP^2)
$$
for any $m\geqslant 2$. The conformal volume $V_c(m,\mathbb RP^2)$ is known to be $6\pi$ for $m\geqslant 4$, see~\cite{LiYau82},  and we get
\begin{equation}
\label{p:case2}
V_c^*(m,\Sigma^2)\leqslant 6\pi\left[\frac{\gamma+3}{2}\right]\qquad\text{for any }\quad m\geqslant 4.
\end{equation}
The combination of relations~\eqref{p:case1} and~\eqref{p:case2}, yields the statement of the theorem.
\qed

\section{Lower bounds for the number of negative eigenvalues of Schr\"odinger operators}
\label{s:schro:proofs}
\subsection{Proof of Theorem~\ref{t:schro}}
We regard a unit sphere $S^m\subset\mathbb R^{m+1}$ as a metric space equipped with the intrinsic distance function $d$. For a given weakly conformal immersion $\phi:\Sigma^n\to S^m$ we denote by $k$ the integer part
$$
\left[\left(\frac{9c}{2500}\right)^{n/2}\frac{1}{n^{n/2}V^*_c(m,\phi)}\frac{1}{\mathit{Vol}_g(\Sigma^n)^{n/2-1}}\left(\int_{\Sigma^n}\mathcal V\mathit{dVol}_g\right)^{n/2}\right],
$$
where the constant $c$ is from Proposition~\ref{ds1} applied to the metric space $(S^m,d)$. For a proof of the theorem it is sufficient to show that there exists $k+1$ Lipschitz test-functions $u_i$ on $\Sigma^n$ whose supports are disjoint and such that
\begin{equation}
\label{punch:schro}
\int_{\Sigma^n}\abs{\nabla u_i}^2\mathit{dVol}_g< \int_{\Sigma^n}\mathcal Vu_i^2\mathit{dVol}_g;
\end{equation}
the latter would imply that
$$
N(\mathcal V)\geqslant k+1\geqslant \frac{C(n,m)}{V^*_c(m,\phi)}\frac{1}{\mathit{Vol}_g(\Sigma^n)^{n/2-1}}\left(\int_{\Sigma^n}\mathcal V\mathit{dVol}_g\right)^{n/2}.
$$
The test-functions $u_i$ are constructed similarly to those in the proof of Theorem~\ref{t:cv}. In more detail, we start with equipping the metric space $(S^m,d)$ with two measures $\mu$ and $\nu$. The measure $\mu$ is a push-forward of the volume measure $\mathit{Vol}_g$ on $\Sigma^n$ under the weakly conformal immersion $\phi:\Sigma^n\to S^m$, and the measure $\nu$ is defined by the formula
$$
\nu(A)=\int_{\phi^{-1}(A)}\mathcal V(x)\mathit{dVol}_g(x)\qquad\text{for any Borel set~ }A\subset S^m.
$$
By Lemma~\ref{l:bis}  for any point $p\in S^m$ the level set $\phi^{-1}(p)$ has zero Lebesgue measure, and hence, the measure $\nu$ is non-atomic. We can apply Proposition~\ref{ds1} to $(S^m,d)$ with the measure $\nu$ to find a collection $\{2A_i\}$ of $2(k+1)$ disjoint annuli such that
\begin{equation}
\label{denom1:schro}
\nu(A_i)\geqslant c\nu(S^m)/(2k+2)\geqslant c\nu(S^m)/(4k).
\end{equation}
Since the annuli $\{2A_i\}$ are disjoint, we also have
$$
\sum_{i=1}^{2k+2}\mu(2A_i)\leqslant\mu(S^m),
$$
and hence, there exists at least $k+1$ sets $2A_i$ such that
\begin{equation}
\label{denom2:schro}
\mu(2A_i)\leqslant\mu(S^m)/(k+1)\leqslant\mu(S^m)/k.
\end{equation}
Without loss of generality, we may assume that these inequalities hold for $i=1,\ldots, k+1$. Denote by $u_i$ the test-functions constructed in Section~\ref{proof:cv}. Each $u_i$ is supported in an annulus $2A_i$, and following the argument in the proof of Theorem~\ref{t:cv}, we have
\begin{multline}
\label{num:schro}
\int_{\Sigma^n}\abs{\nabla u_i}^2d\mathit{Vol}_g\leqslant 4nV^*_c(m,\phi)^{2/n}\mu(2A_i)^{1-2/n}\leqslant 4nV^*_c(m,\phi)^{2/n}(\mu(S^m)/k)^{1-2/n}\\ <9nV^*_c(m,\phi)^{2/n}(\mathit{Vol}_g(\Sigma^n)/k)^{1-2/n},
\end{multline}
where in the second inequality we used relation~\eqref{denom2:schro}. Using Lemmas~\ref{l1} and~\ref{l2} together with relation~\eqref{denom1:schro}, we can also obtain
\begin{equation}
\label{denom3:schro}
\int_{\Sigma^n}\mathcal V u_i^2d\mathit{Vol}_g\geqslant \frac{81}{625}(\nu(A_i))\geqslant\frac{81}{2500}c\nu(S^m)/k,
\end{equation}
where $c$ is a constant from Proposition~\ref{ds1}, which depends only on $m$. Combining relations~\eqref{num:schro} and~\eqref{denom3:schro}, we see that
$$
\left(\int_{\Sigma^n}\abs{\nabla u_i}^2\mathit{dVol}_g\right)/\left(\int_{\Sigma^n}\mathcal V u_i^2\mathit{dVol}_g\right)<\frac{2500}{9c}nV^*_c(m,\phi)^{2/n}\frac{\mathit{Vol}_g(\Sigma^n)^{1-2/n}}{\nu(S^m)}k^{2/n}\leqslant 1,
$$
where in the last inequality we used the definition of the integer $k$. This relation demonstrates~\eqref{punch:schro} and finishes the proof of the theorem.
\qed

\subsection{Proof of Theorem~\ref{t:schro2}}
Following the notation and a circle of ideas discussed in Section~\ref{proof:rei}, we prove the following more general statement.
\begin{theorem}
\label{t:schro3}
Let $(M^m,h)$ be a Riemannian manifold that admits a conformal immersion into a unit sphere $S^m\subset\mathbb R^{m+1}$ of the same dimension $m\geqslant 2$. Let $(\Sigma^n,g)$ be a closed Riemannian manifold of dimension $n\geqslant 2$ that admits an isometric immersion $\phi:\Sigma^n\to M^m$. Then for any non-negative potential $\mathcal V\in L^p(\Sigma^n)$, where $p>n/2$, the number of negative eigenvalues of the Schr\"odinger operator $(-\Delta_\Sigma-\mathcal V)$ satisfies the inequality
$$
N(\mathcal V)\geqslant C(n,m)\left(\int_{\Sigma^n}\mathcal V\mathit{dVol}_{g}\right)/\left(\int_{\Sigma^n}\left(\abs{H_\phi}^2+R_\phi\right)\mathit{dVol}_{g}\right),
$$
where $H_\phi$ is the mean curvature vector of $\phi$, and $R_\phi$ is given by relation~\eqref{n:rphi}.
\end{theorem}
\begin{proof}
We follow closely the line of argument in the proof of Theorem~\ref{t:schro} above. Let $\phi:(\Sigma^n,\phi)\to (M^m,h)$ be a given isometric immersion, and $\Pi:(M^m,h)\to (S^m,g_{\mathit{can}})$ be a conformal immersion. By $\psi$ we denote the composition $\Pi\circ\phi:\Sigma^n\to S^m$. We view a unit sphere $S^m\subset\mathbb R^{m+1}$ as a metric space equipped with the intrinsic distance function $d$ and denote by $k$ the integer part
$$
\left[\frac{9c}{1250n}\left(\int_{\Sigma^n}\mathcal V\mathit{dVol}_g\right)/\left(\int_{\Sigma^n}\left(\abs{H_\phi}^2+R_\phi\right)\mathit{dVol}_{g}\right)\right],
$$
where the constant $c$ is from Proposition~\ref{ds1} applied to the metric space $(S^m,d)$. For a proof of the theorem it is sufficient to show that there exists $k+1$ Lipschitz test-functions $u_i$ whose supports are disjoint and such that
\begin{equation}
\label{punch:schro2}
\int_{\Sigma^n}\abs{\nabla u_i}^2\mathit{dVol}_g< \int_{\Sigma^n}\mathcal V u_i^2\mathit{dVol}_g;
\end{equation}
the latter would imply that
$$
N(\mathcal V)\geqslant k+1\geqslant {C(n,m)}\left(\int_{\Sigma^n}\mathcal V\mathit{dVol}_g\right)/\left(\int_{\Sigma^n}\left(\abs{H_\phi}^2+R_\phi\right)\mathit{dVol}_{g}\right).
$$
We equip the metric space $(S^m,d)$ with the measure $\nu$ that is the push-forward of the measure defined by the density $\mathcal V$ on $\Sigma^n$ under the conformal immersion $\psi:\Sigma^n\to S^m$; in other words,
$$
\nu(A)=\int_{\psi^{-1}(A)}\mathcal V(x)\mathit{dVol}_g(x)\qquad\text{for any Borel set~ }A\subset S^m.
$$
Applying Proposition~\ref{ds1} to the metric space $(S^m,d)$ with the measure $\nu$, we find a collection $\{2A_i\}$ of $k+1$ disjoint annuli such that
\begin{equation}
\label{denom1:schro2}
\nu(A_i)\geqslant c\nu(S^m)/(k+1)\geqslant c\nu(S^m)/(2k).
\end{equation}
Denote by $u_i$ the test-functions constructed in Section~\ref{proof:cv}. Each $u_i$ is supported in an annulus $2A_i$, and following the argument in the proof of Theorem~\ref{t:rei}, we have
\begin{equation}
\label{num:schro2}
\int_{\Sigma^n}\abs{\nabla u_i}^2d\mathit{Vol}_g\leqslant 4n\int_{\Sigma^n}\left(\abs{H_\phi}^2+R_\phi\right)\mathit{dVol}_{g}<9n\int_{\Sigma^n}\left(\abs{H_\phi}^2+R_\phi\right)\mathit{dVol}_{g},
\end{equation}
where in the last inequality we used the fact that the integral can not vanish identically, see Corollary~\ref{c:rei}. Using Lemmas~\ref{l1} and~\ref{l2} together with relation~\eqref{denom1:schro2}, we  obtain
\begin{equation}
\label{denom3:schro2}
\int_{\Sigma^n}\mathcal Vu_i^2d\mathit{Vol}_g\geqslant \frac{81}{625}(\nu(A_i))\geqslant\frac{81}{1250}c\nu(S^m)/k,
\end{equation}
where $c$ is a constant from Proposition~\ref{ds1} and depends only on $m$. Combining relations~\eqref{num:schro2} and~\eqref{denom3:schro2}, we see that
$$
\left(\int_{\Sigma^n}\abs{\nabla u_i}^2\mathit{dVol}_g\right)/\left(\int_{\Sigma^n}\mathcal V u_i^2\mathit{dVol}_g\right)<\frac{1250}{9c}\frac{n}{\nu(S^m)}\left(\int_{\Sigma^n}\left(\abs{H_\phi}^2+R_\phi\right)\mathit{dVol}_{g}\right)k\leqslant 1,
$$
where in the last inequality we used the definition of the integer $k$. This relation yields inequality~\eqref{punch:schro2} for each test-function $u_i$, where $i=1,\ldots, k+1$.
\end{proof}


\end{document}